\documentclass[11pt]{article}
\usepackage[margin=28mm]{geometry}
\usepackage{amsmath,amssymb,amsthm,enumitem,microtype}
\usepackage[hidelinks]{hyperref}
\usepackage{tikz}
\usepackage{graphicx}
\usetikzlibrary{calc}

\newcommand{\avd}{\overline d}

\newcommand{\e}{\mathrm e}

\tikzset{
  auxvertex/.style={
    circle,
    fill=black,
    inner sep=1.15pt
  },
  smallvertex/.style={
    circle,
    fill=black,
    inner sep=.75pt
  }
}

\newcommand{\DrawTone}[2]{%
  \begin{scope}[shift={#1},rotate=#2]
    \path[
      draw=black,
      fill=black,
      line width=.35pt
    ]
      (0,0)--(-.11,1.30)--(.11,1.30)--cycle;

    \node[smallvertex] at (0,0) {};
  \end{scope}
}

\newcommand{\DrawTtwo}[2]{%
  \begin{scope}[shift={#1},rotate=#2]

    \path[
      draw=blue!65!black,
      fill=blue!45,
      fill opacity=.18,
      draw opacity=.85,
      line width=.55pt
    ]
      (0,0)--(-.28,3.20)--(.28,3.20)--cycle;

    \node[auxvertex]   at (0,0) {};
    \node[smallvertex] at (0,.80) {};
    \node[smallvertex] at (0,1.60) {};
    \node[smallvertex] at (0,2.40) {};

    \foreach \y in {0,.80,1.60,2.40}{%
      \DrawTone{(0,\y)}{-30}
      \DrawTone{(0,\y)}{30}
    }

  \end{scope}
}

\newcommand{\DrawTthree}[2]{%
  \begin{scope}[shift={#1},rotate=#2]

    \path[
      draw=black!60,
      fill=black!35,
      fill opacity=.14,
      draw opacity=.75,
      line width=.6pt
    ]
      (0,0)--(-.40,5.10)--(.40,5.10)--cycle;

    \node[auxvertex]   at (0,0) {};
    \node[smallvertex] at (0,3.25) {};

    \DrawTtwo{(0,0)}{-50}
    \DrawTtwo{(0,0)}{50}

    \DrawTtwo{(0,3.25)}{-50}
    \DrawTtwo{(0,3.25)}{50}

    \path[
      draw=black!65,
      fill=black!48,
      line width=.45pt
    ]
      (0,5.10) ellipse [x radius=.40, y radius=.105];

    \node[
      font=\scriptsize,
      text=black,
      inner sep=0pt,
      anchor=south
    ] at (0,5.27) {$B_3^\sigma(v)$};

  \end{scope}
}

\newtheorem{thm}{Theorem}[section]
\newtheorem{lem}[thm]{Lemma}
\newtheorem{claim}[thm]{Claim}
\newtheorem{prop}[thm]{Proposition}
\theoremstyle{definition}
\newtheorem{defn}[thm]{Definition}
\theoremstyle{plain}

\author{
Nemanja Dragani\'c\thanks{Bocconi University, Milan, Italy.  \\\emph{Email}: \textbf{nemanja.draganic@unibocconi.it}.
Research supported in part by the Javotte Bocconi Institute. Parts of this research were carried out while the author was at the University of Oxford and supported by SNSF grant 217926, and while the author at the EPF Lausanne and supported by SNSF grant P5R5-2\_239120.}
}

\title{New tools for sparse graph embeddings: Hybrid expanders and application to nested cycles}

\date{}

\begin{document}
\maketitle

\begin{abstract}
We introduce \emph{hybrid expanders}, sparse graphs which combine sublinear expansion across all scales with strong constant-factor expansion on smaller, polylogarithmic sets. We show that every $n$-vertex graph of average degree at least
$2^{\omega(\log^*n)}$
 contains such a subgraph, and use this local expansion to bring into the sparse setting tree-embedding and rollback tools more commonly associated with random or pseudorandom graphs. As an application, for fixed $k\in\mathbb N$, we prove that every graph
satisfying this average-degree condition contains  $k$ pairwise edge-disjoint cycles with nested vertex sets and compatible cyclic orders, resolving up to a quasi-constant factor a problem of Thomassen, originating in work of Erd\H{o}s.
\end{abstract}
\section{Introduction}
  Many classical problems in extremal graph theory ask how large the average
  degree of a graph must be to force a particular configuration. For several
  problems concerning cycles and subdivisions, a sufficiently large constant
  average degree turns out to suffice, even when additional conditions are
  imposed on lengths or cyclic orders.

A remarkable collection of results of this kind has appeared over the last
decade.  Liu and Montgomery~\cite{LiuMontgomery} proved that sufficiently large constant average
degree forces a cycle whose length is a power of two, resolving a problem of
Erd\H{o}s~\cite{LiuMontgomery}; in the same work they solved the
odd-cycle problem of Erd\H{o}s and Hajnal~\cite{Erdos1981},
and confirmed Thomassen's conjecture by proving that sufficiently large
constant average degree forces a balanced clique subdivision
\cite{Thomassen1984}.  Gil Fern\'andez and Liu~\cite{FernandezLiuPillar}
subsequently proved Thomassen's pillar conjecture~\cite{Thomassen1989}:
sufficiently large constant minimum degree forces two disjoint cycles of the
same length, whose corresponding vertices are connected with vertex-disjoint
paths of the same length.  

A common starting point for such arguments is the sublinear-expander method
introduced by Koml\'os and Szemer\'edi in their work on topological cliques
\cite{KomlosSzemerediI,KomlosSzemerediII}.  Their argument, together with the
independent proof of Bollob\'as and Thomason~\cite{BollobasThomason}, established the quadratic
average-degree threshold for subdivisions of complete graphs.  The method has since developed into a general
paradigm for attacking similar questions: one passes to a subgraph which retains a substantial proportion of
the original average degree and in which all sets expand by a sublinear factor,
and then carries out the desired construction inside this subgraph. The
surveys of Letzter~\cite{LetzterSurvey} and Montgomery~\cite{MontgomerySurvey} give broad accounts of the many results
obtained in this way, and give insight about the intricate machinery developed for sublinear expanders.

Several important questions remain open, although we have seen significant progress on them.
In a question which attracted a lot of recent attention~\cite{DraganicMethukuCorreiaSudakov,DvorakMartinsThomasseTrotignon,DraganicGirao, letzter2026nearly, ChakrabortiJanzerMethukuMontgomery}, Chen,
Erd\H{o}s and Staton~\cite{ChenErdosStaton} asked whether sufficiently large constant average degree
forces a cycle with at least as many chords as vertices.  Dragani\'c, Methuku,
Munh\'a Correia and Sudakov~\cite{DraganicMethukuCorreiaSudakov} proved that average degree $O(\log^8 n)$ suffices. More recently, motivated in part by a
question of Dvo\v{r}\'ak, Martins, Thomass\'e and Trotignon
\cite{DvorakMartinsThomasseTrotignon}, Dragani\'c and Gir\~ao\cite{DraganicGirao} showed that a
large constant minimum degree already forces a cycle with almost linearly
many chords, but the loss remains polylogarithmic in the length of the cycle. Furthermore, in upcoming work, the authors manage to bring the required average degree bound from~\cite{DraganicMethukuCorreiaSudakov} down to essentially $\Theta(\log\log n)$.

\paragraph{Nested cycles.}
The problem studied here originates in a collection of three questions posed
by Erd\H{o}s in 1975 \cite{Erdos1975}.  The first asks for two edge-disjoint
cycles whose vertex sets are nested.  Bollob\'as proved that a linear number
of edges suffices \cite{BollobasNested}, and Chen, Erd\H{o}s and Staton~\cite{ChenErdosStaton}
extended this to any fixed number of cycles.  In the third question, the two cycles must have exactly the same vertex
set. A construction of Pyber, R\"odl and Szemer\'edi gives a lower bound
of order $\log\log n$ on the required average degree
\cite{PyberRodlSzemeredi}, while Chakraborti, Janzer, Methuku and Montgomery
proved the first polylogarithmic upper bound
\cite{ChakrabortiJanzerMethukuMontgomery}.

The second question lies between these two.  Here the vertex sets are nested,
but the common vertices must also occur in the same cyclic order.  Equivalently,
the cycles can be drawn as nested polygons without geometric crossings.  Gil
Fern\'andez, Kim, Kim and Liu proved the optimal linear bound for two such
cycles, using a delicate construction around a shortest cycle in a sublinear
expander \cite{FKKL}. The question asking for $k$ such cycles, asked by
Thomassen~\cite{Thomassen1989} conjectured that, for every fixed $k$,
  sufficiently large constant average degree forces $k$ edge-disjoint
  nested cycles without geometric crossings.

In his recent ICM survey, Montgomery~\cite{MontgomerySurvey} reiterated this question, and observed that the shortest-cycle mechanism
in that proof appears inherently suited to producing one further cycle starting from the shortest one, but
does not offer an evident route to an arbitrary number of cycles.  Recently, Xu, Zeng and Zhang~\cite{XuZengZhang} obtained the first
general
result.  Combining a controlled wrapping argument with
robust sublinear-expander technology developed by Alon, Buci\'c, Sauermann,
Zakharov and Zamir~\cite{AlonBucicSauermannZakharovZamir}, they proved that, for every fixed $k\ge3$, every
sufficiently large $n$-vertex graph with average degree at least
\[
 C_k (\log n)^{k-1}(\log\log n)^{k-3}
\]
contains $k$ such cycles. Subsequently, Ai, Gutin and Hao~\cite{AiGutinHao} sharpened this bound to
\[
 O_k\!\left(\frac{(\log\log n)^2}{\log\log\log n}\right)
\] for every fixed \(k\ge3\).
Our improvement comes within a hair's breadth of the conjectured constant bound.

\begin{thm}\label{thm:nested}
For every integer $k\ge1$ there is a constant $C=C(k)$ such that every
$n$-vertex graph of average degree at least
\[
 2^{C\log^* n}
\]
contains $k$ pairwise edge-disjoint cycles with nested vertex sets and
compatible cyclic orders.
\end{thm}

The main difficulty in passing from two cycles to several is that the
construction must be iterated. The two-cycle argument starts from a
shortest cycle, whose minimality gives useful control over its
neighbourhood. The outer cycle produced by that argument need not retain
this property, so the same argument cannot simply be repeated. In our
construction, every new path is equipped with further trees that allow
another layer to be built around it. These preparations must be maintained
at every vertex while preserving the cyclic order.

\paragraph{Hybrid expanders}
Sublinear expanders are very well suited to finding paths between large vertex sets while avoiding forbidden sets, and this has been one of the main reasons for their effectiveness in sparse extremal graph theory. Their expansion guarantees, however, may be weaker than one could hope for small sets. In the usual application of the method, it produces a graph which expands on all relevant scales, typically for every set of size at most half the order of the graph \cite{KomlosSzemerediI,KomlosSzemerediII,LetzterSurvey}. In this form, the expansion factor cannot in general be taken to be a positive constant. Moshkovitz and Shapira~\cite{MoshkovitzShapira} constructed graphs with large constant average degree for which every subgraph has a set of size comparable to its order $n$, with expansion at most $1/\log n$, up to factors involving $\log\log n$. Thus the inverse-logarithmic loss is essentially unavoidable when one asks for one uniform guarantee over the entire range of set sizes.
The scale-sensitive notion introduced by Koml\'os and Szemer\'edi gives a more refined statement. In its standard formulation, a set $X$ is guaranteed an expansion ratio of order essentially
$
\rho(|X|)=\Theta\left(\frac{1}{\log^2(|X|)}\right)
$~\cite{KomlosSzemerediI,KomlosSzemerediII,LetzterSurvey}. This is much stronger near the bottom of the permitted range, but the guaranteed ratio still tends to zero as $X$ grows. In particular, it does not provide constant-factor expansion throughout a useful range of small sets. 

The main idea of the paper is to separate two tasks which are usually both
assigned to the same sublinear expander. 
The following is the central notion of this paper.
\begin{defn}
For $n,K,T\in \mathbb N$, an $n$-vertex graph is a $(K,T)$\emph{-hybrid expander} if:
\[
 |N_G(X)|\ge \frac{|X|}{\log^2 n}
 \qquad\text{for }1\le |X|\le N/2,
\]
and
\[
 |N_G(X)|\ge K|X|
 \qquad\text{for }1\le |X|\le \log^T n.
\]
\end{defn}
On large scales it retains
the weak expansion that is available in arbitrary sparse graphs and is
sufficient for global routing.  On the polylogarithmic scale on which the
construction takes place, it behaves as a genuine large constant-factor expander.

This local expansion allows us to import tools which are normally available
in random or otherwise random-like graphs.  In particular, we use the
Friedman--Pippenger tree-embedding theorem to prepare many possible routes in
advance \cite{FriedmanPippenger}.  The rollback principle, presented for example in~\cite{DKN,MontgomerySpanningTrees}, then allows unused parts of an embedded tree to be
pruned without losing the embedding invariant.  The weak global
expansion is used afterwards, to connect two large surviving leaf sets
by a short path while avoiding everything constructed earlier.  In this way,
local preparation and global routing are handled by the two different
expansion regimes for which they are respectively best suited. For the proof of Theorem~\ref{thm:nested}, this indeed gives a simple proof for $k=2$. For larger $k$, because of the recursive nature of the construction, we will interchangeably use both types of expansion: every time we do some global routing, we will make sure that the resulting graph is good in the appropriate context, so that we may continue with our local extendability machinery.



The point of the method is to enhance sublinear expanders with stronger expansion at the small scale, which allows for much stronger and more elegant arguments than in their classic applications.  Under the slightly stronger degree assumption used
here, this makes the argument for the problem of two nested cycles with compatible orders particularly transparent
and gives, in a uniform way, the local reservoirs that must be manufactured
more delicately in the work of Liu and Montgomery, Gil Fern\'andez, Kim, Kim
and Liu, and Gil Fern\'andez and Liu
\cite{LiuMontgomery,FKKL,FernandezLiuPillar}. 

\paragraph{Towards constant degree.}
It is natural to ask whether hybrid expansion can be extracted already from
sufficiently large constant average degree.  This is false in general.  We construct $n$-vertex graphs with average degree $\Omega(\log^*n)$ in which every nontrivial subgraph $H$ contains a nonempty set $X$ with
\[
 |X|\le \log(|V(H)|)
 \qquad\text{and}\qquad
 |N_H(X)|<10|X|.
\]
Thus no direct strengthening of our extraction lemma can settle the
constant-degree conjecture.

We believe, however, that this points to a broader programme.  A graph of large average degree either contains a
hybrid expander, where random-like embedding tools become available, or it
fails local expansion repeatedly. The natural hope is to understand this process
well enough to find the desired configuration directly in the graphs that
do not contain a hybrid expander.  In this sense, hybrid expansion could potentially be one side of a potential structural dichotomy for sparse graphs.

\section{Preliminaries}

All graphs are finite and simple, and all logarithms have base two.  We write
$ \log^{(0)} n=n$ and $\log^{(t+1)}n=\log\bigl(\log^{(t)}n\bigr)$,
and let $\log^*n$ be the least $t\ge0$ for which $\log^{(t)}n\le1$.
For $X\subseteq V(G)$, put $
 \Gamma_G(X)=\bigcup_{x\in X}N_G(x)
 $ and 
 $
 N_G(X)=\Gamma_G(X)\setminus X$.

For $\alpha>0$, a graph $G$ is an \emph{$\alpha$-expander} if
\[
 |N_G(X)|\ge\alpha|X|
 \qquad\text{whenever }1\le |X|\le |V(G)|/2.
\]

We say that cycles $C_1,\ldots,C_k$ are \emph{nested} if their edge sets are
pairwise disjoint, their vertex sets satisfy
$
 V(C_1)\subseteq\cdots\subseteq V(C_k),
$
and the cycles can be oriented so that, for every $a<b$, the cyclic order on
$V(C_a)$ induced by $C_a$ is the restriction of the cyclic order induced by
$C_b$.

\subsection{The Friedman-Pippenger extendability method}
We briefly recall the tree-embedding method of Friedman and Pippenger, usually referred to as the extendability method. Building on the embedding results of Friedman–Pippenger [12] and Haxell [15], this method allows one to embed trees in suitable expanding graphs. 
Traditionally, this method was mostly utilized in random and random-like graphs, where the trees we can embed are usually of size linear in the size of the graph. More precisely, if sets up to order $s=s(n)$ expand well, then one can find bounded degree trees of order $s$. 

We will use Theorem~\ref{thm:fp} as our main embedding tool. We also record the simple Lemma~\ref{lem:pruning}, which shows that leaves may be removed from an extendable tree while preserving extendability. This makes it possible to apply Theorem~\ref{thm:fp} iteratively. Variants of this method have been used in the resolution of several long-standing problems in graph theory; see, for example,~\cite{DKN, MontgomerySpanningTrees, Draganic2024Hamiltonicity}
\begin{defn}
Let $s\in\mathbb N$ and $K>0$.  A graph $G$ is $(s,K)$-expanding if
$|N_G(X)|\ge K|X|$ for every $X\subseteq V(G)$ with $|X|\le s$.
\end{defn}

\begin{defn}
Let $s,D\in\mathbb N$ and let $\phi:F\hookrightarrow G$ be an embedding.  We
write $\phi(F)$ for the vertex image $\phi(V(F))$, and put
$d_F(\phi^{-1}(v))=0$ when $v\notin\phi(F)$.  The embedding $\phi$ is
$(s,D)$-good if $\Delta(F)\le D$ and, for every $X\subseteq V(G)$ with
$|X|\le s$,
\begin{equation}\label{eq:extendable}
 |\Gamma_G(X)\setminus\phi(F)|\ge
 \sum_{v\in X}\bigl(D-d_F(\phi^{-1}(v))\bigr)+|\phi(F)\cap X|.
\end{equation}
\end{defn}

When $F\subseteq G$, we say that $F$ is $(s,D)$-good in $G$ if the
inclusion embedding is $(s,D)$-good.

We use the Friedman--Pippenger extension theorem~\cite{FriedmanPippenger}
and the pruning lemma in the form proved by Dragani\'c, Krivelevich, and
Nenadov~\cite{DKN}.

\begin{thm}[Friedman--Pippenger extension]\label{thm:fp}
Let $s,D$ be positive integers, let $F$ have maximum degree at most $D$ and
fewer than $s$ vertices, and let $\phi:F\hookrightarrow G$ be
$(2s-2,D)$-good.  If $G$ is $(2s-2,D+2)$-expanding, then every graph $F'$
obtained from $F$ by successively adding leaves, with $|V(F')|\le s$ and
$\Delta(F')\le D$, has an extending $(2s-2,D)$-good embedding in $G$.
\end{thm}

\begin{lem}[Pruning]\label{lem:pruning}
For positive integers $s,D$, if $\phi:F\hookrightarrow G$ is $(s,D)$-good
and $F'$ is obtained from $F$ by successively deleting leaves, then the
restriction of $\phi$ to $F'$ remains $(s,D)$-good in $G$.
\end{lem}

\section{Hybrid expanders}

Throughout this section, constants hidden in $O_k(\cdot)$ may depend on
$k$.  We write $G-S=G[V(G)\setminus S]$ and
$\overline d(G)=2e(G)/|V(G)|$.

\subsection{An expanding subgraph}

We first pass to a subgraph which expands strongly on polylogarithmic sets
and still has enough expansion on all sets of size at most half its order.

\begin{lem}[Hybrid expander subgraph]\label{lem:expander-subgraph}
For every $K,T\ge1$ there is a constant $C=C(K,T)$ such that the following
holds.  If an $n$-vertex graph has average degree at least $d$, where
$
 d\ge 2^{C\log^*n},
$
then it contains a subgraph $G$ of order $N\ge d^{1/10}$ satisfying
\begin{align}
 |N_G(X)|&\ge |X|/\log^2N
     &&\text{for }1\le |X|\le N/2,\label{eq:bulk-expansion}\\
 |N_G(X)|&\ge K|X|
     &&\text{for }1\le |X|\le \log^T N.\label{eq:local-expansion}
\end{align}
Equivalently, $G$ is a $(K,T)$-hybrid expander.
\end{lem}

\begin{proof}
Starting with the given graph, we will repeatedly pass to a subgraph until we either found the required subgraph, or the size is about to drop below $d^{1/10}$.  Suppose the
current graph is $H$, with $m$ vertices and average degree $a$, and that one
of the two desired inequalities fails for a nonempty set $X$.  If
$\overline d(H-X)\ge a$, replace $H$ by $H-X$.  Otherwise, putting
$Y=N_H(X)$, there are no edges from $X$ to
$V(H)\setminus(X\cup Y)$, and hence
\[
 2e(H[X\cup Y])>a|X|.
\]
We then replace $H$ by $H[X\cup Y]$.

If $X$ violates~\eqref{eq:bulk-expansion}, the new graph has at most
$(1+\log^{-2}m)m/2\le 3m/4$ vertices and average degree at least
$\frac{a|X|}{|X\cup Y|}\geq (1-\log^{-2}m)a$.  If $X$ violates~\eqref{eq:local-expansion}, the new
graph has at most
\[
 f(m):=(K+1)\log^T m
\]
vertices and average degree at least $\frac{a|X|}{|X\cup Y|}\geq a/(K+1)$.  For large $m$,
$f(f(m))\le\log m$, so the second type of replacement occurs at most
$2\log^*n+O_{K,T}(1)$ times before the order falls below $d^{1/10}$.

At the first type of replacement the successive orders decrease by a factor
at most $3/4$.  Therefore, for the subgraph sizes $m_1<m_2<m_3\ldots <m_t$ it holds that
\[
 \sum_{i\in[t]}\frac1{\log^2m_i}\leq \sum_{j\geq 0}\frac{1}{(\log(d^{1/10})+j\log(\frac{4}{3}))^2}=O(1/\log d)
 \quad\text{and hence}\quad
 \prod_i\left(1-\frac1{\log^2m_i}\right)\ge\frac12.
\]
Thus every graph produced by the procedure has average degree at least
\[
 d \prod_i\left(1-\frac1{\log^2m_i}\right)(K+1)^{-2\log^*n-O_{K,T}(1)}\ge\frac{\sqrt d}{2},
\]
provided $C(K,T)$ is sufficiently large.  The process cannot reach a graph
with fewer than $d^{1/10}$ vertices, since its average degree would then
exceed its order.  It must therefore stop at a graph satisfying
both expansion conditions.
\end{proof}

\subsection{Routing around a good preparation}

We shall use without further comment the following monotonicity: if an
$(s,D)$-good graph has maximum degree at most $D'$, then it is also
$(s',D')$-good whenever $s'\le s$ and $D'\le D$.

The next lemma is the only routing statement used in the construction.
The lemma allows us to join two sets $L_1,L_2$ in a good preparation
while avoiding a prescribed set $Z$. After deleting $Z$ and the rest of
the preparation, we remove a further small set $R$ to restore local
expansion. The global expansion of $G$ then gives a short connecting
path, and a shortest such path is good in the remaining graph.

\begin{lem}[Routing lemma]\label{lem:routing}
Let $G$ be an $N$-vertex $\alpha$-expander
for some $\alpha>0$.
Let $H$ be an induced subgraph of $G$, let $F\subseteq H$ be
$(s,D)$-good in $H$ with $\Delta(F)\le\delta$, and let
$L_1,L_2\subseteq V(F)$ be disjoint.  Let
\(
 D'=D-\delta-12 
\) and
suppose that $D'\ge2$, $1\le t\le N/2$. Finally,  let
$Z\subseteq V(H)\setminus V(F)$ such that
\begin{align}
 t+|Z|/9&\le s,\label{eq:routing-local}\\
 |V(G)\setminus V(H)|+|V(F)|+10|Z|/9
   &\le \alpha t/4.\label{eq:routing-global}
\end{align}
Then there is a set $R$, with $|R|\le |Z|/9$, such that
\[
 H_0=H-\bigl((V(F)\setminus(L_1\cup L_2))\cup Z\cup R\bigr)
\]
is $(t,D'+3)$-expanding.  If $|L_1|,|L_2|>|Z|/9$, then $H_0$ contains an
$L_1$--$L_2$ path of length at most
\(
 O\bigl(\log t+\alpha^{-1}\log N\bigr).
\)
A shortest such path is $(t,D')$-good in $H_0$.
\end{lem}

\begin{proof}
Put
$
 J=H-(V(F)\setminus(L_1\cup L_2)).
$
For $X\subseteq V(J)$ with $|X|\le s$, goodness gives
\begin{align*}
 |N_J(X)|
 &\ge |\Gamma_H(X)\setminus V(F)|-|X\setminus V(F)|\\
 &\ge (D-1)|X\setminus V(F)|
   +\sum_{x\in X\cap(L_1\cup L_2)}(D-d_F(x)+1)\\
 &\ge (D-\delta)|X|.
\end{align*}
Thus $J$ is $(s,D-\delta)$-expanding.

Put $R_0=\varnothing$. If $J-(Z\cup R_{i-1})$ contains a nonempty
set $X_i$ of size at most $t$ with fewer than $(D'+3)|X_i|$
external neighbours, put $R_i=R_{i-1}\cup X_i$.
Every vertex of $N_J(R_i)\setminus Z$ belongs to the external
neighbourhood of some $X_a$ at the step when $X_a$ was removed.
Since the sets $X_a$ are disjoint, it follows that
\[
 |N_J(R_i)|<|Z|+(D'+3)|R_i|.
\]
Whenever $|R_i|\le s$, expansion in $J$ therefore gives
\[
 (D-\delta)|R_i|
 \le |N_J(R_i)|
 <|Z|+(D'+3)|R_i|,
\]
and hence $9|R_i|<|Z|$.

Suppose that $i$ is the first step at which $|R_i|>|Z|/9$.
Then
\[
 |R_i|\le |Z|/9+t\le s,
\]
so the preceding inequality applies and gives a contradiction.
Thus the process stops with a set $R$ satisfying $|R|\le |Z|/9$.
At termination, $H_0=J-(Z\cup R)$ is $(t,D'+3)$-expanding.

Let
\[
 B=(V(G)\setminus V(H))\cup(V(F)\setminus(L_1\cup L_2))\cup Z\cup R.
\]
For every $X\subseteq V(G-B)$ with $t\le |X|\le N/2$,
\eqref{eq:routing-global} gives
\[
 |N_{G-B}(X)|\ge\alpha|X|-|B|\ge 3\alpha|X|/4.
\]
Since $|L_i|>|R|$, each $L_i$ contains a vertex of $H_0$.
Grow a ball in $H_0$ from one such vertex in each set.
While a ball has fewer than $t$ vertices, local expansion increases
its size by a factor of at least $D'+4$, so it reaches size at least
$t$ after $O(\log t)$ steps. Once its size is between $t$ and $N/2$,
the preceding inequality increases its size by a factor of at least
$1+3\alpha/4$. Thus both balls exceed $N/2$ vertices after a further
$O(\alpha^{-1}\log N)$ steps. They must intersect, giving an
$L_1$--$L_2$ path of the required length.

Let $P$ be a shortest $L_1$--$L_2$ path in $H_0$.
The path $P$ has no chord. Moreover, a vertex outside $P$ cannot
have neighbours at positions differing by more than two along $P$,
since these would give a shorter path with the same endpoints.
Consequently, every vertex of $H_0$ has at most three neighbours on $P$.
For $|X|\le t$, we therefore have
\[
 |\Gamma_{H_0}(X)\setminus V(P)|
 \ge |N_{H_0}(X)|-3|X|
 \ge D'|X|.
\]
Since $L_1$ and $L_2$ are disjoint, every vertex of $P$ has degree
at least one in $P$. Taking $d_P(v)=0$ outside $V(P)$, this gives
\[
 \sum_{v\in X}\bigl(D'-d_P(v)\bigr)+|X\cap V(P)|
 \le D'|X|.
\]
Together with $\Delta(P)\le2\le D'$, this proves that $P$ is
$(t,D')$-good in $H_0$.
\end{proof}

\section{Nested cycles without geometric crossings}
Before turning to the argument for general $k\ge 2$ nested cycles, we treat the case $k=2$. This result will not be used in the proof of the general case; we include it only to introduce the reader to the Friedman--Pippenger machinery and to illustrate the basic utility of hybrid expanders.

We recall that the case $k=2$ was already resolved in \cite{FKKL}, where it was proven that a constant average-degree bound suffices. That proof is more delicate and relies on robust sublinear expanders in their usual form. By contrast, the argument below gives a quick and transparent proof under the slightly weaker assumption of quasi-constant average degree. We hope it makes clear why hybrid expanders are particularly well suited to this type of construction.

\subsection{Warm-up: two nested cycles}\label{subsec:two-cycles}

\begin{prop}\label{prop:two-cycles}
There are absolute constants $K_0,T_0$ such that every $N$-vertex
$(K_0,T_0)$-hybrid expander contains two nested cycles.
\end{prop}

\begin{proof}
Take $D=20$, $K_0=23$ and $T_0=40$. Put
$q=\lceil\log^{10}N\rceil$, $s=\lfloor\log^{30}N\rfloor$ and
$\lambda=\lceil c\log^4N\rceil$, where $c$ is a sufficiently large
constant.

Let $C$ be a shortest cycle in $G$. Since $G$ has minimum degree at least
three, $|C|\le2\log N+2$. Moreover, every vertex has at most three
neighbours on $C$. Hence, for
$|X|\le2s-2,$\[
|\Gamma_G(X)\setminus V(C)|\ge(K_0-3)|X|=D|X|,\]
and so $C$ is
$(2s-2,D)$-good. At every vertex
of $C$, attach two disjoint rooted binary trees, each having between $q$ and $2q$ leaves and fewer than
$4q$ vertices; hence we can think of it as each edge getting one attached tree for each of its vertices. The resulting graph has fewer than $s$ vertices and maximum
degree at most four. Theorem~\ref{thm:fp} embeds it in $G$ while preserving
$(2s-2,D)$-goodness; denote the current embedded graph by $F$.

We now process the edges of $C$ one by one, in each round producing a short path between the vertices of the current edge. Suppose that the current
edge is $xy$. Let $L_x$ and $L_y$ be the leaf sets in the trees assigned to
this edge, and let $U$ consist of the internal vertices of the connecting
paths chosen for earlier edges which lie outside $F$. At every stage
$|U|=O(\lambda \log N)$, while $|V(F)|=O(q\log N)=o(s/\log^2N)$.
Lemma~\ref{lem:routing}, applied with $t=s$, gives a short path $P$ from a
leaf $a\in L_x$ to a leaf $b\in L_y$, avoiding $U$ and the rest of $F$, thus producing a new short $x-y$ path.
Repeat the same procedure for the next edge.

The replacement paths for distinct edges are internally disjoint: their
branches lie in different incidence trees, while each new middle path avoids
both the current preparation and the middle paths chosen earlier. After all
edges have been processed, concatenate the replacement paths in the cyclic
order of $C$. They form a cycle $C'$. The cycles $C$ and $C'$ are
edge-disjoint, $V(C)\subseteq V(C')$, and the vertices of $C$ occur on $C'$
in their original cyclic order.
\end{proof}

The proof above reserves two large sets of choices at every incidence and
uses a short route to join the appropriate pair. For more layers, the only
new requirement is that every selected branch, and every new middle path,
should carry the same structure one rank lower. This leads to the recursive
trees below.

\subsection{$k$ nested cycles}
{\textit{Proof of Theorem~\ref{thm:nested}}.}
Before we proceed with our construction, we introduce several essential parameters, and pass to a hybrid expander.  
We set $ \delta=2k+2$, which will be the maximum degree of our construction throughout the process. Next, 
$
 D_r=\delta+10+r(\delta+12)
$
for $r\geq0$ will be the expansion rate of of the host graph at depth $r$ (it will become clear later what depth refers to).

Thus
\begin{equation}\label{eq:D-recursion}
 D_r-\delta-12=D_{r-1}\qquad(r\ge1).
\end{equation}
We apply Lemma~\ref{lem:expander-subgraph} to our $n$-vertex host graph of average degree $2^{c_k\log^*n}$ with
$
 K=D_k+3,
 $ and $
 T=100k^2
 $ and obtain a $(K,T)$-hybrid expander. Furthermore, we stress that the precise definition of $D_k,T, \delta$ is not particularly important: the reader can think about them as large constants depending only on $k$.
For the rest of the proof let $G$ be an $N$-vertex graph satisfying
\begin{align}
 |N_G(X)|&\ge \alpha|X| &&(1\le |X|\le N/2),
       &\alpha&=\log^{-2}N,\label{eq:global-expansion-G}\\
 |N_G(X)|&\ge K|X| &&(1\le |X|\le\log^T N).
       &&\label{eq:strong-expansion-G}
\end{align}

\subsubsection{Auxiliary trees and prepared paths}

\noindent\textbf{The base tree.} Put $ q=\lceil\log^{10k}N\rceil$.
Let $B(q)$ be a perfect rooted binary tree of height $\lceil\log q\rceil$;
it has between $q$ and $2q$ leaves and fewer than $4q$ vertices. Note that it is important that the leaf set of this tree is a large power of $\log N$: this will be important for connecting two such leaf-sets while avoiding a smaller set of used vertices (in applications of Lemma~\ref{lem:routing}).

\begin{defn}[Auxiliary trees]\label{def:auxiliary-trees}
The signs $-$ and $+$ distinguish two isomorphic copies, reserved for
the incoming and outgoing edges of an oriented path, respectively.
For $\sigma\in\{-,+\}$, define $T_r^\sigma(v)$ recursively
(see Fig.~\ref{fig:auxiliary-trees}).  Put
$T_0^\sigma(v)=\{v\}$.  For $r\ge1$, start with a copy of
$T_{r-1}^\sigma(v)$, join $v$ to the root of a fresh copy
$B_r^\sigma(v)$ of $B(q)$, and at every vertex $z$ of
$B_r^\sigma(v)$ attach a fresh copy of
\[
 T_{r-1}(z)=T_{r-1}^-(z)\cup T_{r-1}^+(z),
\]
identifying only the root $z$.  The two trees $T_r^-(v)$ and
$T_r^+(v)$ meet only at $v$, and we write
\[
 T_r(v)=T_r^-(v)\cup T_r^+(v).
\]
The leaves of $B_r^\sigma(v)$ are denoted by $L_r^\sigma(v)$.
\end{defn}
A direct induction gives $ |T_r(v)|\le(10q)^r$, and $\Delta(T_r(v))\le2r+1$.
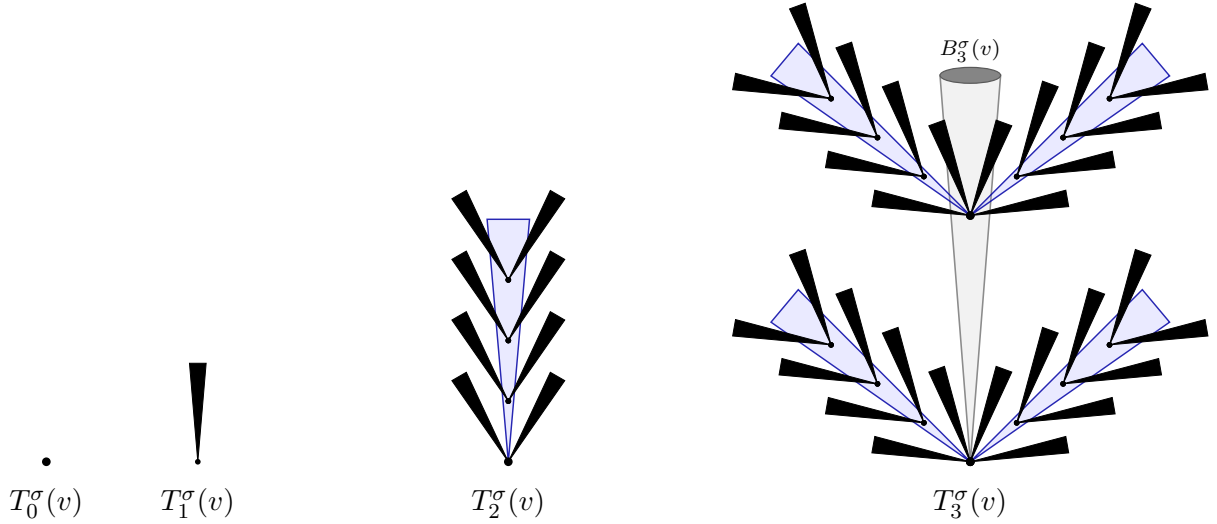
\begin{figure}[ht]
\centering
\resizebox{\textwidth}{!}{%
\begin{tikzpicture}

\begin{scope}[shift={(0,0)}]
  \node[auxvertex] at (0,0) {};
  \node at (0,-.55) {$T_0^\sigma(v)$};
\end{scope}

\begin{scope}[shift={(2.0,0)}]
  \DrawTone{(0,0)}{0}
  \node at (0,-.55) {$T_1^\sigma(v)$};
\end{scope}

\begin{scope}[shift={(6.1,0)}]
  \DrawTtwo{(0,0)}{0}
  \node at (0,-.55) {$T_2^\sigma(v)$};
\end{scope}

\begin{scope}[shift={(12.2,0)}]
  \DrawTthree{(0,0)}{0}
  \node at (0,-.55) {$T_3^\sigma(v)$};
\end{scope}

\end{tikzpicture}%
}

\caption{
The first four auxiliary trees.
Every triangle introduced at each stage represents a copy of \(B(q)\).
At each stage, to every vertex of \(B(q)\) represented by the vertical triangle, two copies of the previous tree are attached.
}
\label{fig:auxiliary-trees}
\end{figure}

\begin{defn}[Preparation]\label{def:preparation}
 If $R=v_0v_1\cdots v_\ell$ is a path, where we also allow
$v_0=v_\ell$ and then regard $R$ as a cycle with distinguished vertex
$v_0$, an \emph{$r$-preparation} of $R$ is the graph
\begin{equation}\label{eq:canonical-preparation}
 A_r(R)=R\cup\bigcup_{i=0}^{\ell-1}
 \bigl(T_r^+(v_i)\cup T_r^-(v_{i+1})\bigr),
\end{equation}
where all trees $T^\sigma_r(v_i)$ are pairwise vertex disjoint, except that $T_r^+(v_i)\cap T_r^-(v_i)=\{v_i\}$.
\end{defn}

\begin{claim}\label{claim:branch-preparation}
Let $r\ge1$.  If $a\in L_r^\sigma(x)$, then the $x$--$a$ path in
$T_r^\sigma(x)$, oriented from $x$ to $a$, has an $(r-1)$-preparation obtained
by consecutively removing leaves from $ T_r^\sigma(x)$. 
\end{claim}

\begin{proof}
This is immediate from the construction.
\end{proof}

\begin{defn}[Nested path system]\label{def:nested-path-system}
A \emph{nested path system of depth $r$} from $u$ to $v$ is a tuple
\[
 \mathcal P=(P^0,P^1,\ldots,P^r)
\]
of pairwise edge-disjoint oriented $u$--$v$ paths such that
$V(P^i)\subseteq V(P^{i+1})$, and the vertices of $P^i$ occur on
$P^{i+1}$ in the same order, for every $i<r$.  When $u=v$, these paths are
cycles, read from the distinguished vertex $u$.
We write $V(\mathcal P)=\bigcup_{i=0}^rV(P^i)$.
\end{defn}

\subsubsection{Building path systems recursively}
 We first define further parameters we will use in the proof. Again, the specific choices of functions which depend on $k$ are not so important; their relative order is. We also repeat the definition of $q$ to have all non-constant parameters in one place, for the convenience of the reader:
\[
 \lambda=\lceil c_k\log^4N\rceil,
 \qquad
 s_0=\left\lfloor\frac{\log^T N}{4^{k+2}}\right\rfloor,
 \qquad
 s_r=4^rs_0\quad(0\le r\le k),
\]
\[
 b_r=2(4\lambda)^r,
 \qquad
 \mu=20\lambda(10q)^k,\qquad q=\lceil\log^{10k}N\rceil.
\]
Every preparation arising below will have at most $\mu/4$ vertices.  The path
provided by Lemma~\ref{lem:routing} has fewer than $\lambda/2$ edges, and a
root-to-leaf path in $B(q)$ has fewer than $\lambda/8$ edges.  These
statements follow directly from the definitions.

The following is the main recursive lemma.  
Its statement says that, given an edge $x,y$ in a good embedding, which also has $T_r^-(x), T^+(y)$ attached to $x,y$ respectively, one can use those trees to construct a path system of depth $r$ between $x,y$, which is not too large and avoids a prescribed not so large set $U$. 

\begin{lem}[Path system for one edge]\label{lem:prepared-edge}
Let $0\le r\le k$, let $H\subseteq G$ be induced, and let $F\subseteq H$ be
$(2s_r-2,D_r)$-good with $\Delta(F)\le\delta$ and $|V(F)|\le\mu/4$.
Suppose that $xy\in E(F)$ and that $F$ contains disjoint copies of $T_r^+(x)$ and
$T_r^-(y)$, attached to the rest of $F$ only at $x$ and $y$, respectively.
Let $U\subseteq V(H)\setminus V(F)$, and suppose that
\[
 |V(G)\setminus V(H)|\le(k-r)\mu
 \qquad\text{and}\qquad
 |U|+b_r\le q/100.
\]
Then there is a nested path system $\mathcal P=(P^0,\ldots,P^r)$ from $x$
to $y$, with $P^0=xy$ and $|V(\mathcal P)|\le b_r$, such that
\[
 (V(\mathcal P)\setminus\{x,y\})\cap U=\varnothing
 \quad\text{and}\quad
 V(\mathcal P)\cap V(F)
 \subseteq V(T_r^+(x))\cup V(T_r^-(y)).
\]
\end{lem}

\begin{proof}
We argue by induction on $r$.  For $r=0$, take $P^0=xy$.

Let $r\ge1$, and assume the statement at rank $r-1$.  Put
$L_x=L_r^+(x)$ and $L_y=L_r^-(y)$.  We apply
Lemma~\ref{lem:routing} with $Z=U$, $s=2s_r-2$, $t=2s_{r-1}-2$, and
$D=D_r$.  
Condition~(\ref{eq:routing-local}) of the lemma follows from $s_r=4s_{r-1}$ and
$q=o(s_0)$,
since 
\[
t+|Z|/9\leq (2s_{r-1}-2)+q/100\leq s_r/2+o(s_r)\leq s
\]
For Condition~(\ref{eq:routing-global}) by the assumptions of the lemma we directly get
\[
 |V(G)\setminus V(H)|+|V(F)|+10|U|/9
 \le (k-r)\mu+\mu/4+q
 =o(\alpha s_{r-1}),
\]
as $q\ll\mu \ll s_0\leq s_{r-1}$.
Also $|L_x|,|L_y|\ge q>|U|/9$.  Thus there are
$a\in L_x$ and $b\in L_y$ joined by a shortest path $P$ in
\begin{equation}\label{eq:middle-host}
 H_P=H-\bigl((V(F)\setminus(L_x\cup L_y))\cup U\cup R_0\bigr),
 \qquad |R_0|\le |U|/9.
\end{equation}
Since $P$ is a shortest path, it only intersects $L_1$ and $L_2$ once, precisely at its endpoints, say at $a\in L_1$ and $b\in L_2$.  
It has fewer
than $\lambda/2$ edges and is
$(2s_{r-1}-2,D_{r-1})$-good in $H_P$.

Let $Q_x$ be the $x$--$a$ path in $T_r^+(x)$,
and let $Q_y$ be the $b$--$y$ path in $T_r^-(y)$.  By Claim~\ref{claim:branch-preparation}, each of those paths carries the
preparation required at rank $r-1$.

We first complete $P$.  Since $H_P$ is
$(2s_{r-1}-2,D_{r-1}+3)$-expanding, Theorem~\ref{thm:fp} extends $P$ to a
good copy $F_P$ of $A_{r-1}(P)$ in $H_P$.  Moreover,
\[
 |V(G)\setminus V(H_P)|
 \le (k-r)\mu+|V(F)|+|U|+|R_0|
 \le (k-r+1)\mu.
\]
Apply the induction hypothesis successively to the edges of $P$, each time
forbidding the systems already constructed on its preceding edges.  The
forbidden set in every application lies outside $F_P$, and its size,
together with the remaining construction, is at most
$|E(P)|b_{r-1}<\lambda b_{r-1}<q/100$.  We obtain a nested path system
$\mathcal P_P$ of depth $r-1$ from $a$ to $b$, using at most
$|E(P)|b_{r-1}$ vertices.

Next prune $T_r^+(x)$ to the branch $Q_x$ and its rank-$(r-1)$ preparation,
leaving the rest of $F$ unchanged.  Let $F_x$ be the resulting graph.  By
Lemma~\ref{lem:pruning}, $F_x$ is still good, and by monotonicity it is good
with the rank-$(r-1)$ parameters.  Put
\[
 U_x=(U\cup V(\mathcal P_P))\setminus V(F_x).
\]
Applying the induction hypothesis successively along $Q_x$ gives a nested
path system $\mathcal P_x$ of depth $r-1$.  Indeed,
\[
 |U_x|+|E(Q_x)|b_{r-1}
 \le |U|+(|E(P)|+|E(Q_x)|)b_{r-1}<q/100.
\]
The system $\mathcal P_x$ avoids $\mathcal P_P$ except at $a$: vertices of
$\mathcal P_P$ outside $F_x$ belong to $U_x$, while $H_P$ meets the
rank-$(r-1)$ preparation of $Q_x$ only at $a$.

Similarly, prune $T_r^-(y)$ to the preparation of $Q_y$, leaving the rest
of $F$ unchanged, and denote the resulting good graph by $F_y$.  With
\[
 U_y=(U\cup V(\mathcal P_P)\cup V(\mathcal P_x))\setminus V(F_y),
\]
the induction hypothesis applied successively along $Q_y$ gives a nested
path system $\mathcal P_y$ of depth $r-1$.  The total length of the three
base paths is less than $ \lambda/8+\lambda/2+\lambda/8<\lambda$,
so all the forbidden-set and size conditions used above follow from
$|U|+b_r\le q/100$ and $\lambda b_{r-1}<b_r$.

Now, write the three systems as
$
 \mathcal P_x=(Q_x^0,\ldots,Q_x^{r-1}),
 \mathcal P_P=(P^0,\ldots,P^{r-1}),$ and $
 \mathcal P_y=(Q_y^0,\ldots,Q_y^{r-1})
$.
For $1\le i\le r$, concatenate
$ Q_x^{i-1}P^{i-1}Q_y^{i-1}.
$
Together with $xy$ at level zero, these paths form a nested path system
$\mathcal P$ of depth $r$ from $x$ to $y$.  The three constituent systems
meet only at $a$ and $b$, so the resulting paths are simple and pairwise
edge-disjoint.  Finally,
\[
 |V(\mathcal P)|
 \le 2+(|E(Q_x)|+|E(P)|+|E(Q_y)|)b_{r-1}
 <b_r.
\]
The construction avoids $U$.  Moreover, $H_P$ meets $F$ only in
$L_x\cup L_y$, while the two side systems use vertices of $F$ only inside
their corresponding incidence trees.  This proves the required avoidance
of the rest of $F$.
\end{proof}

\subsection{Proof of Theorem~\ref{thm:nested}}

\begin{proof}
Choose $K,T$ as above and then choose the constant in the theorem large
enough that Lemma~\ref{lem:expander-subgraph} gives a graph $G$ satisfying
\eqref{eq:global-expansion-G} and~\eqref{eq:strong-expansion-G}.  We may
assume that its order $N$ is large enough for all the preceding estimates.

Let $C$ be a shortest cycle in $G$.  The local expansion condition gives
minimum degree at least three, and a breadth-first search yields $ |C|\le2\log N+2<\lambda.$ Every vertex has at most three neighbours on $C$, for otherwise one obtains
a shorter cycle.  Hence, for $|X|\le2s_k-2$,
\[
 |N_G(X)\setminus V(C)|\ge(K-3)|X|=D_k|X|.
\]
It follows directly from~\eqref{eq:extendable} that $C$ is
$(2s_k-2,D_k)$-good.

Orient $C$ and attach the trees prescribed by $A_k(C)$.  This graph is
obtained from $C$ by successively adding leaves, has maximum degree at most
$\delta$, and has fewer than $\mu/4<s_k$ vertices.  Since
$2s_k-2\le\log^T N$ and $K\ge D_k+2$, the
Friedman--Pippenger theorem embeds it in $G$ while preserving
$(2s_k-2,D_k)$-goodness.
Let $F\subseteq G$ denote the resulting embedded copy of $A_k(C)$.
  We keep this preparation fixed throughout the following iteration.
Clearly, $|V(F)|\leq |C||T_k(v)|\leq 10\log n\cdot(10q)^k=o(\mu)$.

Choose a vertex $u$ of $C$ and read the oriented cycle as a closed
$u$--$u$ path. 
We now iteratively can use Lemma~\ref{lem:prepared-edge} on its edges, one by one. Indeed, suppose we constructed path systems $\mathcal P_1,\mathcal P_2,\ldots, \mathcal P_j$, each of depth $k$, for the first $j<|C|$ edges of $C$, such that the collection of all these paths is pairwise internally disjoint, and the total number of vertices in them is at most $jb_k$. 

Then we apply Lemma~\ref{lem:prepared-edge} to the embedding of $A_k(C)$ in $H=G$, with $r=k$ and $U=\bigcup_{i\in[j]}V(\mathcal P_i)\setminus V(F)$. Note that \[|U|+b_r\leq (j+1)b_k=\Theta(\log^{4k}N)\leq  q/100=\Theta(\log^{10k}N),\]
so all conditions are satisfied, and thus we get a nested path system on the $(j+1)$-th edge, which avoids the earlier path systems and vertices of the preparation of the remaining part of the path, and is of size at most $b_k$. 
Hence we can continue applying the same lemma in the same way to the $k$-preparation of the path. 

Thus every edge gets a nested path system of depth $k$. Concatenating the paths at the same level of each path system, produces a nested cycle system for $C$ of depth $k$, and completes the proof. 
\end{proof}

\section{Graphs with no hybrid expander subgraph}
    Here we exhibit a construction that shows that growing degree does not force a hybrid expander subgraph. In particular, $d=\Omega(\log^* n)$ is already too weak to guarantee expansion of order $10$ even of sets of size $\log |H|$, where $|H|$ is the order of the subgraph.

\begin{thm}\label{thm:bad expansion}
There is a sequence of graphs $G_k$ with $n_k=|V(G_k)|$ and average degree
$\avd(G_k)\geq \frac16\log^* n_k$, and every subgraph $H\subseteq G_k$ with
$h:=|V(H)|\geq 10$ contains a nonempty set $S\subseteq V(H)$ of size
$|S|\leq \log h$ with expansion $|N_H(S)|< 10|S|$. In particular, no such
$H$ is a $(1,10)$-hybrid expander.
\end{thm}

The idea is to build the graph in levels. Each
level will be a matching, which makes the average degree grow, while every
union of levels $i, i+1,\ldots $ will remain sparse on sets up to a suitable size.

We first record the simple observation that allows us to add edges between far vertices
without destroying local sparsity.

\begin{claim}\label{lem:local sparsity}
Let $R$ be a graph and let $L\geq 1$ be an integer. Suppose that
$e(R[U])\leq 3|U|-1$ for every nonempty set $U\subseteq V(R)$ of size at
most $L$. If $x,y\in V(R)$ satisfy
$\operatorname{dist}_R(x,y)\geq L$, then the same property holds in $R+xy$.
\end{claim}

\begin{proof}
Suppose that adding $xy$ creates a set $U$ of size at most $L$ for which
the property fails. Then $x,y\in U$ and $e(R[U])\geq 3|U|-1$. If
$U_1,\dots,U_s$ are the vertex sets of the components of $R[U]$, then
$e(R[U])\leq \sum_{a=1}^s(3|U_a|-1)=3|U|-s$. Thus $s=1$, so $R[U]$ is
connected and $\operatorname{dist}_R(x,y)\leq |U|-1<L$, a contradiction.
\end{proof}

\paragraph{The construction.}
We first construct the graphs together with the decomposition into
matchings that will be used later.

\begin{lem}\label{lem:construction}
There are integers
$1=n_0<n_1<n_2<\cdots$ and, for every $k\geq 1$, a graph $G_k$ on
$n_k$ vertices, together with a linear order on $V(G_k)$ so that $G_k$ is a union of edge-disjoint matchings $\mathcal M_1\cup\mathcal M_2\cup\ldots\cup \mathcal M_k$ and:

\begin{enumerate}

\item For every $0\leq r\leq k$, the ordered vertex set $V(G_k)$ can be
partitioned into consecutive blocks of size $n_r$ in such a way that every
edge from $\mathcal M_i$ with $i<r$, has both endpoints in the same block;

\item For $1\leq j\leq k$, let $Q_j^{(k)}= \mathcal M_j\cup\mathcal M_{j+1}\cup\ldots\cup\mathcal M_k$. Then every nonempty
set $U\subseteq V(G_k)$ of size
$|U|\leq \lceil \e^{n_{j-1}}\rceil$ satisfies
\[
    e(Q_j^{(k)}[U])\leq 3|U|-1.
\]

\item We have $\avd(G_k)\geq k/2\geq (\log^*n_k)/6$ and $\log^*n_k\leq 3k$.
\end{enumerate}
\end{lem}

\begin{proof}
We recursively construct ordered graphs whose edges are coloured by
positive integers. In $G_i$, the edges of each colour will form a
matching. Whenever we take a copy of an earlier graph, we copy both its
vertex order and its edge colours.

Let $n_0=1$, let $n_1=2$, and let $G_1=K_2$, with its unique edge
coloured $1$. For $i\geq 1$, put
$\ell_i:=\lceil \e^{n_{i-1}}\rceil$. For $1\leq j\leq i$, let
$Q_j^{(i)}$ denote the spanning subgraph of $G_i$ formed by the edges
whose colours belong to $\{j,j+1,\dots,i\}$.

We shall ensure that
\begin{equation}\label{eq:local sparsity}
 e(Q_j^{(i)}[U])\leq 3|U|-1
\end{equation}
for every $1\leq j\leq i$ and every nonempty
$U\subseteq V(G_i)$ of size at most $\ell_j$. Thus, after discarding all
colours smaller than $j$, every set of size at most $\ell_j$ remains
sparse.

This clearly holds for $G_1$. Suppose now that $i\geq 2$ and that
$G_{i-1}$ has already been constructed. Let $X_i$ and $Y_i$ each be the
disjoint union of $2i^{\ell_{i-1}}$ ordered, edge-coloured copies of
$G_{i-1}$, and let $F_i$ be their disjoint union. Each of $X_i$ and
$Y_i$ therefore has
$m_i:=2i^{\ell_{i-1}}n_{i-1}$ vertices, while $F_i$ has $2m_i$ vertices.
For every old colour, its edges in $F_i$ form a matching, since they are
the union of the corresponding matchings in pairwise disjoint copies of
$G_{i-1}$.

Starting with an empty matching $M$, greedily add edges between unmatched
vertices $x\in X_i$ and $y\in Y_i$ whenever
$\operatorname{dist}_{F_i+M}(x,y)\geq\ell_{i-1}$, and stop when no such
pair remains.

At every stage of this procedure the current graph has maximum degree at
most $i$: the old edges are the union of $i-1$ matchings, and $M$ is
itself a matching. Hence every ball of radius $\ell_{i-1}-1$ has at most
$1+i+\cdots+i^{\ell_{i-1}-1}\leq i^{\ell_{i-1}}$ vertices.

Suppose that, when the procedure stops, more than
$i^{\ell_{i-1}}$ vertices remain unmatched on each side. Choose an
unmatched vertex $x\in X_i$. Since the ball of radius
$\ell_{i-1}-1$ around $x$ contains at most $i^{\ell_{i-1}}$ vertices,
some unmatched vertex of $Y_i$ lies outside this ball. It could be joined
to $x$, contrary to the maximality of $M$. Thus at most
$i^{\ell_{i-1}}$ vertices remain unmatched on either side, and so
$|M|\geq m_i-i^{\ell_{i-1}}\geq m_i/2$.

Define $G_i:=F_i+M$, retaining all the old edge colours and colouring the
edges of $M$ with the new colour $i$. The old colour classes remain
matchings, and the new colour class is the matching $M$. Moreover,
$n_i:=|V(G_i)|=2m_i=4n_{i-1}i^{\ell_{i-1}}$.

Order first the copies of $G_{i-1}$ belonging to $X_i$, then those
belonging to $Y_i$, always using the inherited order inside each copy.
Iterating this rule, for every $0\leq r\leq i$ the natural copies of
$G_r$ form a partition of $V(G_i)$ into consecutive blocks of size
$n_r$, where for $r=0$ these are simply the singleton blocks. Every edge
of colour less than $r$ has both endpoints in one such copy of $G_r$.
This gives the required block property.

It remains to verify \eqref{eq:local sparsity}. Fix $j\leq i-1$. Before
the edges of $M$ are added, the graph formed by the colours
$j,j+1,\dots,i-1$ is a disjoint union of copies of $Q_j^{(i-1)}$. If a
nonempty set $U$ of size at most $\ell_j$ meets $s$ of these copies in
the nonempty sets $U_1,\dots,U_s$, then, by induction,
$e(Q_j^{(i-1)}[U_1])+\cdots+e(Q_j^{(i-1)}[U_s])
\leq 3|U|-s\leq 3|U|-1$.

Now consider one step of the greedy procedure. Let $M'$ be the matching
chosen so far, and suppose that $xy$ is the next edge to be added. For
each $j\leq i-1$, consider the graph consisting of the old edges of
colours $j,\dots,i-1$ together with the edges of $M'$. This graph is a
subgraph of $F_i+M'$, so the distance between $x$ and $y$ in it is at
least
$\operatorname{dist}_{F_i+M'}(x,y)\geq\ell_{i-1}\geq\ell_j$.
Claim~\ref{lem:local sparsity} therefore shows that adding $xy$ preserves
\eqref{eq:local sparsity}. Applying this argument one edge at a time
proves the property for every $j\leq i-1$. For $j=i$, the graph
$Q_i^{(i)}=M$ is a matching, so the property is immediate.

Since $G_i$ is the union of $i$ matchings, we have $\Delta(G_i)\leq i$.
Moreover, $F_i$ is a disjoint union of copies of $G_{i-1}$ and hence has
the same average degree as $G_{i-1}$. Since $n_i=2m_i$, it follows that
$\avd(G_i)=\avd(G_{i-1})+2|M|/n_i
=\avd(G_{i-1})+|M|/m_i\geq\avd(G_{i-1})+\frac12$.
Consequently, $\avd(G_i)\geq i/2$.

Finally, we control the growth of $n_i$. Put $x:=n_{i-1}$. The recurrence
implies $i\leq x$, and
$\ell_{i-1}=\lceil\e^{n_{i-2}}\rceil\leq 2\e^x$. Hence, very crudely,
\[
 n_i=4x i^{\ell_{i-1}}
 \leq 4x x^{2\e^x}
 \leq \e^{\e^{3x}}.
\]
Thus three applications of $\log$ reduce $n_i$ to at most
$x=n_{i-1}$: indeed,
$\log\log\log n_i\leq\log(3x)\leq x$. Therefore
$\log^*n_i\leq\log^*n_{i-1}+3$, and hence
$\log^*n_k\leq 3k$. Together with $\avd(G_k)\geq k/2$, this gives
$\avd(G_k)\geq(\log^*n_k)/6$.

For a fixed $k$, let $\mathcal M_j$ be the set of edges of colour $j$ in
$G_k$. These colour classes are edge-disjoint matchings whose union is
$G_k$, and the preceding construction gives properties~(1)--(3).
\end{proof}

\noindent\textbf{Proof that the construction works.}
We are ready to show the main result of this section.
\begin{proof}[Proof of Theorem~\ref{thm:bad expansion}]
For each $k\geq 1$, let $G_k$ be the graph given by
Lemma~\ref{lem:construction}. By property~(3),
$\avd(G_k)\geq (\log^*n_k)/6$.

It remains to prove the assertion about expansion. Fix $k$, and let
$H\subseteq G_k$ have vertex set $A$ and order $h:=|A|\geq 10$. 
Choose the largest $r\in\{0,\dots,k\}$ such that
$n_r\leq\log h$, and put $t:=n_r$. 
List the vertices of $A$ in the order inherited from $G_k$, regard this
order cyclically, and consider the $h$ cyclic intervals of length $t$.
Every vertex of $A$ belongs to exactly $t$ of these intervals.

We first estimate the neighbours arising from the low-index matchings.
By property~(1), every edge in
  $\mathcal M_1\cup\cdots\cup\mathcal M_{r-1}$ has endpoints
  at distance at most $t-1$ in the inherited order on $A$.
  Consequently, for any cyclic interval $W$ of length $t$,
  its external neighbours through these matchings lie among
  the $t-1$ vertices immediately preceding or immediately
  following $W$. These matchings therefore contribute at
  most $2(t-1)\le2t$ external neighbours.

The matchings $\mathcal M_r$ and
$\mathcal M_{r+1}$ each contribute at most another $t$ external
neighbours.

It remains to control the matchings of index at least $r+2$. Suppose
first that $r\leq k-2$, and put $Q:=Q_{r+2}^{(k)}$. By the maximality of
$r$, we have $n_{r+1}>\log h$, and hence
$h<\e^{n_{r+1}}\leq\lceil\e^{n_{r+1}}\rceil$. Property~(2), applied with
$j=r+2$ and $U=A$, therefore gives $e(Q[A])\leq 3h-1$.

Since every vertex of $A$ lies in exactly $t$ of the cyclic intervals,
the average of $\sum_{v\in W}d_{Q[A]}(v)$ over all such intervals $W$ is
\[
 \frac{t}{h}\sum_{v\in A}d_{Q[A]}(v)
 =\frac{2t}{h}e(Q[A])<6t.
\]
There is consequently an interval $W$ for which
$\sum_{v\in W}d_{Q[A]}(v)<6t$. In particular, the matchings
$\mathcal M_{r+2},\dots,\mathcal M_k$ contribute fewer than $6t$
external neighbours of $W$. If $r=k-1$, there are no such matchings,
and we may take any cyclic interval $W$ of length $t$.

Combining the three estimates, we obtain
\[
 |N_{G_k[A]}(W)|<2t+t+t+6t=10t.
\]
Since $|W|=t=n_r\leq\log h$  and $|N_A(W)|<10|W|$, we conclude that $H$ is not a
$(1,10)$-hybrid expander.
\end{proof}

\section{Concluding remarks}
\label{sec:concluding-remarks}

Under the average-degree assumption $2^{C\log^* n}$, the methods developed
here give short proofs of the existence of pillars
\cite{FernandezLiuPillar}, cycles whose lengths are powers of two, and
balanced subdivisions of any fixed graph~\cite{LiuMontgomery}.
We sketch how hybrid expansion provides the additional control over
path lengths needed for these applications.

These applications require some additional control over path lengths. It is
convenient to construct gadgets that allow these lengths to be adjusted;
see the \emph{adjusters} introduced by Liu and
Montgomery~\cite[Section~4]{LiuMontgomery}. Such gadgets can be constructed
in a straightforward way using hybrid expansion and the same embedding and
routing tools developed here. One first passes to a bipartite subgraph and
then extracts a hybrid expander. A shortest even cycle gives two paths
between suitable vertices whose lengths differ by two. Attaching rooted
trees and joining $m$ such pieces in succession gives a choice of
paths with lengths $\ell,\ell+2,\ldots,\ell+2m$, allowing the necessary
length corrections while respecting parity.

\vspace{1 cm}

\textbf{Acknowledgments.} We are grateful to Domagoj Brada\v{c} for suggesting the key idea behind the proof of Theorem~\ref{thm:bad expansion}: to equip the graph with an ordering and use edges whose lengths, at each vertex, grow at a tower-type rate. The remaining principal ideas in the paper are due to the author. We also acknowledge the use of GPT-5.6 Sol, which identified and helped correct a technical issue in an earlier version of the manuscript; the correction required a non-trivial application of the machinery developed in the paper.

GPT-5.6 Sol was additionally used for typesetting and in the preparation of some parts of the proofs: in these instances, we supplied  detailed descriptions and drafts of the intended arguments, which were then written up with its assistance. The author gratefully acknowledges OpenAI for providing access to ChatGPT through the ChatGPT for Academic Researchers program.

\textbf{Note.} This work was obtained independently of Ai, Gutin and Hao. A revised version of their preprint claims a constant average-degree bound.

\end{document}